\documentclass[11pt]{amsart}

\usepackage[a4paper,margin=2.5cm]{geometry}
\usepackage{amsmath,amssymb,amsfonts,amsthm}
\usepackage{mathtools}
\usepackage{enumitem}
\usepackage{tikz-cd}
\renewcommand{\le}{\leqslant}
\renewcommand{\ge}{\geqslant}

\newtheorem{theorem}{Theorem}[section]
\newtheorem{lemma}[theorem]{Lemma}

\newtheorem{corollary}[theorem]{Corollary}

\theoremstyle{definition}
\newtheorem{definition}[theorem]{Definition}

\theoremstyle{remark}
\newtheorem{remark}[theorem]{Remark}

\theoremstyle{plain}
\newtheorem*{theoremA}{Theorem A}
\newtheorem*{theoremB}{Theorem B}
\newtheorem*{corollarystar}{Corollary}

\numberwithin{equation}{section}

\newcommand{\R}{\mathbb{R}}
\newcommand{\C}{\mathbb{C}}
\newcommand{\N}{\mathbb{N}}

\newcommand{\id}{\mathrm{Id}}

\newcommand{\BM}{\mathrm{BM}}
\newcommand{\dist}{d_{\BM}}

\begin{document}

\title[{$(\lambda^+)$-injective Banach spaces}]%
{$(\lambda^+)$-injective Banach spaces}

\author[T.~Kania]{Tomasz Kania}
\address[T.~Kania]{Mathematical Institute\\Czech Academy of Sciences\\%
\v{Z}itn\'a 25\\115~67 Praha 1\\Czech Republic
and Institute of Mathematics and Computer Science\\%
Jagiellonian University\\{\L}ojasiewicza 6, 30-348 Krak\'{o}w, Poland}
\email{kania@math.cas.cz, tomasz.marcin.kania@gmail.com}
\thanks{IM CAS (RVO 67985840). }

\author[G.~Lewicki]{Grzegorz Lewicki}
\address[G.~Lewicki]{Institute of Mathematics and Computer Science\\%
Jagiellonian University\\{\L}ojasiewicza 6\\30-348 Krak\'{o}w, Poland}
\email{Grzegorz.Lewicki@im.uj.edu.pl}


\date{}

\begin{abstract}
In a companion paper (\emph{Studia Math.}, 2023), we proved for every
$\lambda\in(1,2]$ the existence of a $(\lambda^+)$-injective renorming
of $\ell_\infty$ that is not $\lambda$-injective, thereby establishing
a~forgotten theorem of Pe{\l}czy\'nski in that range.
The complementary range $\lambda\in(2,\infty)$ was left open.

In the present paper, we resolve this remaining case: for every
$\lambda>2$ we construct a Banach space that is
$(\lambda^+)$-injective but not $\lambda$-injective, completing
Pe{\l}czy\'nski's theorem for all $\lambda>1$.  The construction uses
a single device: the `zero-sum' subspace
$\Sigma_N(Y)\subset Z_\infty^N$,
which multiplies the relative projection constant by
$\mu_N=2-2/N$ while preserving non-attainment.  Iterating this
operation reduces the problem to the range $(1,2]$ already covered
by the companion paper.  Since the ambient spaces arising in the
iteration are finite $\ell_\infty$-sums of $\ell_\infty$, the
resulting examples may be realised as subspaces of~$\ell_\infty$.

We also prove that if two Banach spaces are each isometrically
isomorphic to their own square and each is isometric to a
$1$-complemented subspace of the other, then their Banach--Mazur
distance is at most $9+6\sqrt{3}$.  Consequently, we obtain the
estimate $\dist(L_\infty[0,1],\ell_\infty)\le 9+6\sqrt{3}$,
thereby improving a recent result of Korpalski and Plebanek.
\end{abstract}

\maketitle

\section{Introduction}

Isbell and Semadeni \cite{IsbellSemadeni} proved that every
infinite-dimensional $1$-injective Banach space contains a hyperplane
that is $(2+\varepsilon)$-injective for every $\varepsilon > 0$, yet is
not $2$-injective, and remarked in a footnote that Pe{\l}czy\'nski had
proved for every $\lambda > 1$ the existence of a
$(\lambda + \varepsilon)$-injective space ($\varepsilon > 0$) that is
not $\lambda$-injective.  Unfortunately, no trace of the proof of
Pe{\l}czy\'nski's result has been preserved.

In our companion paper \cite{KaniaLewicki}, we established that result
for $\lambda\in(1,2]$ by constructing an appropriate renorming of
$\ell_\infty$.  The case $\lambda\in(1,2)$ used hyperplane kernels of
singular functionals on $\ell_\infty$ combined with results of
Baronti--Franchetti \cite{BarontiF} and Blatter--Cheney
\cite{BlatterCheney}; the endpoint $\lambda=2$ was treated via
norm-attainment properties of singular functionals, recovering the
original theorem of Isbell--Semadeni.

At the end of \cite{KaniaLewicki} we remarked that the complementary
range $\lambda\in(2,\infty)$ remained a mystery to us.  The hyperplane
technique is inherently limited to $\lambda\le 2$ because the projection
constant of any hyperplane in a Banach space is at most $2$.  To go
beyond this barrier, one must work with subspaces of higher (but still
finite) codimension.

In the present paper, we develop a different and entirely explicit
approach.  Given a closed subspace $Y$ of a $1$-injective space $Z$
with relative projection constant $\alpha\in(1,2]$ not attained (as
produced by~\cite{KaniaLewicki}), we introduce the
\emph{zero-sum subspace}
\[
\Sigma_N(Y):=\Bigl\{(y_1,\dots,y_N)\in Y_\infty^N:
\sum_{i=1}^N y_i=0\Bigr\}
\subset Z_\infty^N.
\]
The key result (Lemma~\ref{lem:multiplication}) is that
$\lambda(\Sigma_N(Y),Z_\infty^N)=\mu_N\cdot\lambda(Y,Z)$, where
$\mu_N=2-2/N$, and---crucially---non-attainment of the infimum is
preserved.  Since $\mu_N<2$, the factor $\mu_N$ can be iterated: after
$m$ applications the relative projection constant equals
$\mu_N^m\cdot\alpha$, which can be made to match any prescribed
$\lambda>2$ by a suitable choice of $N$ and~$m$.

Background on injective and separably injective spaces can be found
in \cite{AvilesCabelloCastilloGonzalezMoreno}; the characterisation
of $1$-injective spaces as ranges of contractive projections on
$\ell_\infty(\Gamma)$ is due to Kelley \cite{Kelley}.

Combining with the results of \cite{KaniaLewicki}, we obtain the
following complete resolution of Pe{\l}czy\'nski's theorem.

\begin{theoremA}
For every $\lambda>1$ there exists a Banach space that is
$(\lambda^+)$-injective but not $\lambda$-injective.
\end{theoremA}

The case $\lambda=1$ is special: over $\R$, Lindenstrauss
\cite{Lindenstrauss} proved that every $(1^+)$-injective real Banach
space is $1$-injective.  We do not discuss the corresponding
complex-scalar problem here that remains open to date.

Our second result concerns Banach--Mazur distances.
Recall that for isomorphic Banach spaces $X$ and $Y$ the
\emph{Banach--Mazur distance} is
\[
\dist(X,Y)
:=\inf\{\|T\|\,\|T^{-1}\|: T\colon X\to Y\text{ is an isomorphism}\}.
\]
A Banach space $X$ is \emph{isometrically square} if
$X\cong X\oplus_\infty X$ via a surjective linear isometry.

\begin{theoremB}
Let $X$ and $Y$ be Banach spaces such that
\begin{enumerate}[label=\textup{(\roman*)}]
\item $X$ is isometric to a $1$-complemented subspace of $Y$ and
$Y$ is isometric to a $1$-complemented subspace of $X$;
\item $X$ and $Y$ are both isometrically square.
\end{enumerate}
Then $\dist(X,Y)\le 9+6\sqrt{3}$.
\end{theoremB}

\begin{corollarystar}
$\dist(L_\infty[0,1],\ell_\infty)\le 9+6\sqrt{3}\approx 19.39$.
\end{corollarystar}

Since $(3+\sqrt{2})^2=11+6\sqrt{2}\approx 19.49 > 9+6\sqrt{3}$, this
improves the explicit upper bound of Korpalski and
Plebanek~\cite{KorpalskiPlebanek}.

\subsection*{Organisation}

Unless explicitly stated otherwise, Banach spaces are considered over a
fixed scalar field $\mathbb K\in\{\R,\C\}$.
Sections~\ref{sec:prelim}--\ref{sec:main_proof} develop the proof of
Theorem~A.  Section~\ref{sec:optimised} proves Theorem~B.


\section{Preliminaries: injectivity and projection constants}%
\label{sec:prelim}

We write
\[
\mathcal P(F,E):=\{P\colon F\to E : P|_E=\id_E\}
\]
for the set of bounded projections from a Banach space $F$ onto a closed
subspace $E\subset F$, and
\[
\lambda(E,F):=\inf\{\|P\|: P\in\mathcal P(F,E)\}
\]
for the \emph{relative projection constant} of $E$ in~$F$.

\begin{definition}\label{def:injective}
Let $X$ be a Banach space and $\lambda\ge 1$.
\begin{enumerate}[label=(\roman*)]
\item $X$ is \emph{$\lambda$-injective} if for every Banach space $Z$,
every subspace $E\subset Z$, and every bounded operator $T:E\to X$,
there exists an extension $\widetilde T:Z\to X$ with
$\|\widetilde T\|\le \lambda\|T\|$.
\item $X$ is \emph{$(\lambda^+)$-injective} if it is
$(\lambda+\varepsilon)$-injective for all $\varepsilon>0$.
\end{enumerate}
\end{definition}

\begin{figure}[t]
\[
\begin{tikzcd}[column sep=large,row sep=large]
E \arrow[r, hook, "\iota"] \arrow[dr, "T"'] &
Z \arrow[d, dashed, "\widetilde T"] \\
& X
\end{tikzcd}
\qquad
\|\widetilde T\|\le \lambda\|T\|.
\]
\caption{The extension property defining $\lambda$-injectivity:
every bounded operator $T\colon E\to X$ defined on a subspace
$E\subset Z$ extends across the inclusion $\iota$ with norm loss at
most the factor $\lambda$.}
\label{fig:lambda-injectivity}
\end{figure}

\begin{lemma}\label{lem:sum-injective}
Let $X_1,\dots,X_n$ be $1$-injective Banach spaces.  Then the finite
$\ell_\infty$-sum $X_1\oplus_\infty\cdots\oplus_\infty X_n$ is
$1$-injective.
\end{lemma}

\begin{proof}
Let $E\subset F$ be Banach spaces and let
$T=(T_1,\dots,T_n)\colon E\to X_1\oplus_\infty\cdots\oplus_\infty X_n$
be bounded.  Since each $X_j$ is $1$-injective, there exists an
extension $\widetilde T_j\colon F\to X_j$ of $T_j$ with
$\|\widetilde T_j\|=\|T_j\|\le\|T\|$.  Define
$\widetilde T:=(\widetilde T_1,\dots,\widetilde T_n)$.
Then $\widetilde T$ extends $T$ and
$\|\widetilde T\|=\max_j\|\widetilde T_j\|\le\|T\|$.
\end{proof}

\begin{lemma}\label{lem:from-proj-to-inj}
Let $Z$ be a $1$-injective Banach space and let $Y\subset Z$ be a closed
subspace with $a:=\lambda(Y,Z)<\infty$.  If the infimum defining $\lambda(Y,Z)$
is not attained, then $Y$ is $(a^+)$-injective but not $a$-injective.
\end{lemma}

\begin{proof}
Fix $\varepsilon>0$.  By definition of $a$, there exists
$P_\varepsilon\in\mathcal P(Z,Y)$ with
$\|P_\varepsilon\|<a+\varepsilon$.  Let $E\subset F$ be Banach spaces
and let $T\colon E\to Y$ be bounded.  Since $Z$ is $1$-injective, there
exists $\widetilde T\colon F\to Z$ extending $T$ with
$\|\widetilde T\|=\|T\|$.  Then
$\widehat T:=P_\varepsilon\widetilde T\colon F\to Y$ extends $T$
(because $P_\varepsilon|_Y=\id_Y$) and
$\|\widehat T\|\le\|P_\varepsilon\|\,\|T\|<(a+\varepsilon)\|T\|$.
Since $\varepsilon>0$ was arbitrary, $Y$ is $(a^+)$-injective.

If $Y$ were $a$-injective, applying the definition to the pair
$Y\subset Z$ and the identity $\id_Y\colon Y\to Y$ would yield
$P\in\mathcal P(Z,Y)$ with $\|P\|\le a$, contradicting non-attainment.
\end{proof}

\section{The zero-sum construction}\label{sec:zerosum}

For a non-zero Banach space $X$ and an integer $n\ge 2$, we write
$X_\infty^n:=X^n$ equipped with the norm
$\|(x_1,\dots,x_n)\|_\infty:=\max_{1\le i\le n}\|x_i\|$
and define the \emph{zero-sum subspace}
\[
\Sigma_n(X):=\Bigl\{(x_1,\dots,x_n)\in X_\infty^n :
\sum_{i=1}^n x_i=0\Bigr\}.
\]
We also define the \emph{centring projection}
\[
S_n^X(x_1,\dots,x_n)
:=\Bigl(x_i-\frac{1}{n}\sum_{j=1}^n x_j\Bigr)_{i=1}^n
\]
and put $\mu_n:=2-2/n$.

\begin{lemma}\label{lem:Sn}
Let $X\neq\{0\}$ be a Banach space and let $n\ge 2$.  Then $S_n^X$ is
a projection from $X_\infty^n$ onto $\Sigma_n(X)$ and
$\|S_n^X\|=\mu_n$.
\end{lemma}

\begin{proof}
The sum of the coordinates of $S_n^Xx$ equals
$\sum_i x_i-\sum_j x_j=0$, so $S_n^Xx\in\Sigma_n(X)$.  If
$x\in\Sigma_n(X)$, then $\sum_j x_j=0$, hence $S_n^Xx=x$.  Therefore
$S_n^X$ is a projection onto~$\Sigma_n(X)$.

For the upper bound, let $M:=\max_i\|x_i\|$.  Then, for each $i$,
\begin{align*}
\Bigl\|x_i-\frac{1}{n}\sum_{j=1}^n x_j\Bigr\|
&=\Bigl\|\frac{n-1}{n}x_i-\frac{1}{n}\sum_{j\ne i}x_j\Bigr\|\\
&\le \frac{n-1}{n}\|x_i\|+\frac{1}{n}\sum_{j\ne i}\|x_j\|
\le \frac{n-1}{n}M+\frac{n-1}{n}M
=\mu_n M.
\end{align*}
Taking the maximum over $i$ gives $\|S_n^Xx\|_\infty\le\mu_n\|x\|_\infty$.

For the reverse inequality, pick $u\in X$ with $\|u\|=1$ and set
$x:=(u,-u,\dots,-u)\in X_\infty^n$.  Then $\|x\|_\infty=1$ and
$\sum_j x_j=(2-n)u$.  The first coordinate of $S_n^Xx$ equals
$u-(2-n)u/n=\mu_n u$, and each remaining coordinate equals
$-u-(2-n)u/n=-2u/n$.  Hence
$\|S_n^Xx\|_\infty=\mu_n$, giving $\|S_n^X\|\ge\mu_n$.
\end{proof}

The following lemma is the core mechanism of the paper.

\begin{lemma}
\label{lem:multiplication}
Let $Y\subset Z$ be Banach spaces with
$a:=\lambda(Y,Z)<\infty$, and assume that the infimum defining $\lambda(Y,Z)$
is not attained.  Fix $n\ge 2$ and put
$W:=\Sigma_n(Y)\subset Z_\infty^n$.  Then
\[
\lambda(W,Z_\infty^n)=\mu_n\,a,
\]
and the infimum defining $\lambda(W,Z_\infty^n)$ is not attained.
\end{lemma}

\begin{proof}
\emph{Upper bound.}
Fix $\varepsilon>0$ and choose $Q_\varepsilon\in\mathcal P(Z,Y)$ with
$\|Q_\varepsilon\|<a+\varepsilon$.  Define the coordinatewise operator
$\widehat Q_\varepsilon\colon Z_\infty^n\to Y_\infty^n$,
$\widehat Q_\varepsilon(z_1,\dots,z_n):=
(Q_\varepsilon z_1,\dots,Q_\varepsilon z_n)$,
so that $\|\widehat Q_\varepsilon\|=\|Q_\varepsilon\|$.  Set
$P_\varepsilon:=S_n^Y\circ\widehat Q_\varepsilon\colon Z_\infty^n\to W$.
Since $S_n^Y$ projects onto $\Sigma_n(Y)=W$, the operator
$P_\varepsilon$ maps into~$W$.  If $w\in W$, then $w\in Y_\infty^n$,
$\widehat Q_\varepsilon w=w$, and $S_n^Yw=w$; hence
$P_\varepsilon\in\mathcal P(Z_\infty^n,W)$.  By Lemma~\ref{lem:Sn},
\[
\|P_\varepsilon\|
\le \|S_n^Y\|\,\|\widehat Q_\varepsilon\|
=\mu_n\|Q_\varepsilon\|
<\mu_n(a+\varepsilon).
\]
Since $\varepsilon>0$ is arbitrary, $\lambda(W,Z_\infty^n)\le\mu_n a$.

\smallskip
\emph{Lower bound and non-attainment.}
Let $P\in\mathcal P(Z_\infty^n,W)$ be arbitrary.  For a permutation
$\sigma\in\mathfrak S_n$, let
$U_\sigma\colon Z_\infty^n\to Z_\infty^n$ be the coordinate permutation
$U_\sigma(z_1,\dots,z_n):=(z_{\sigma^{-1}(1)},\dots,z_{\sigma^{-1}(n)})$.
Each $U_\sigma$ is a surjective linear isometry with $U_\sigma(W)=W$.
Define the symmetrisation
\[
\widetilde P:=\frac{1}{n!}\sum_{\sigma\in\mathfrak S_n}
U_\sigma^{-1}\,P\,U_\sigma.
\]
Then $\widetilde P\in\mathcal P(Z_\infty^n,W)$,
$\|\widetilde P\|\le\|P\|$, and $\widetilde P$ commutes with every
coordinate permutation.

For $z\in Z$ and $1\le j\le n$, let $e_j(z)\in Z_\infty^n$ denote the
vector with $z$ in the $j$-th coordinate and zero elsewhere.  Define
operators $A,B\colon Z\to Y$ by declaring
$\widetilde P(e_1(z))=(Az,Bz,\dots,Bz)$ for $z\in Z$.
This is possible because $\widetilde P$ is invariant under permutations
that fix the first coordinate, so coordinates $2,\dots,n$ of
$\widetilde P(e_1(z))$ must be equal.  Since
$\widetilde P(e_1(z))\in W\subset Y_\infty^n$, both $Az$ and $Bz$
lie in~$Y$.

For any $j$, choosing $\sigma$ with $\sigma(1)=j$ gives
$e_j(z)=U_\sigma e_1(z)$ and
$\widetilde P(e_j(z))=U_\sigma\widetilde P(e_1(z))$, so
$\widetilde P(e_j(z))$ has $Az$ in the $j$-th coordinate and $Bz$ in
every other.  By linearity, for $z=(z_1,\dots,z_n)$ and each $i$,
$(\widetilde Pz)_i=Az_i+B\sum_{j\ne i}z_j$.
Since $\widetilde Pz\in W$, the sum of its coordinates vanishes:
$(A+(n-1)B)\sum_i z_i=0$ for all $z$, hence $A+(n-1)B=0$.  On the
other hand, $e_1(y)-e_2(y)\in W$ for $y\in Y$, and $\widetilde P$
fixes~$W$, so looking at the first coordinate gives $Ay-By=y$.

Setting $R:=A-B\colon Z\to Y$, we have $R\in\mathcal P(Z,Y)$.  Solving
$A+(n-1)B=0$ and $A-B=R$ yields $A=\frac{n-1}{n}R$ and
$B=-\frac{1}{n}R$, whence
$(\widetilde Pz)_i = R(z_i-\frac{1}{n}\sum_j z_j)$,
i.e.\ $\widetilde P=\widehat R\circ S_n^Z$, where
$\widehat R(z_1,\dots,z_n):=(Rz_1,\dots,Rz_n)$ satisfies
$\|\widehat R\|=\|R\|$.  By Lemma~\ref{lem:Sn},
$\|\widetilde P\|\le\|R\|\,\mu_n$.

We claim equality holds.  Fix $\eta>0$ and choose $u\in Z$ with
$\|u\|=1$ and $\|Ru\|>\|R\|-\eta$.  For $x:=(u,-u,\dots,-u)$, the
proof of Lemma~\ref{lem:Sn} gives
$S_n^Zx=(\mu_n u,-\tfrac{2}{n}u,\dots,-\tfrac{2}{n}u)$, hence
$\|\widetilde P\|\ge\|\widetilde Px\|_\infty=\mu_n\|Ru\|
>\mu_n(\|R\|-\eta)$.
Letting $\eta\downarrow 0$ yields $\|\widetilde P\|=\mu_n\|R\|$.

Since $R\in\mathcal P(Z,Y)$ and the infimum $\lambda(Y,Z)=a$ is not
attained, $\|R\|>a$.  Therefore
$\|P\|\ge\|\widetilde P\|=\mu_n\|R\|>\mu_n a$.
Because $P$ was arbitrary, every projection from $Z_\infty^n$ onto $W$
has norm strictly bigger than $\mu_n a$.  Combined with the upper bound,
$\lambda(W,Z_\infty^n)=\mu_n a$ and the infimum is not attained.
\end{proof}

\section{Proof of Theorem~A}\label{sec:main_proof}

The case $\lambda\in(1,2]$ is treated in \cite{KaniaLewicki}; we invoke
the result we need.

\begin{theorem}[{\cite[Theorem~A]{KaniaLewicki}}]%
\label{thm:KL}
For every $\alpha\in(1,2]$ there exists a closed subspace
$Y_0\subset\ell_\infty$ with $\lambda(Y_0,\ell_\infty)=\alpha$ such that the
infimum defining the projection constant is not attained.
\end{theorem}

\begin{theorem}\label{thm:main_real}
For every $\lambda>2$ there exists a Banach space that is
$(\lambda^+)$-injective but not $\lambda$-injective.
\end{theorem}

\begin{proof}
Fix $\lambda>2$.  Choose an integer $m\ge 1$ with
$2^m\le\lambda<2^{m+1}$, and then an integer $N\ge 3$ large enough that
\[
\mu_N^m>\frac{\lambda}{2},
\qquad\text{where}\quad \mu_N:=2-\frac{2}{N}.
\]
This is possible because $\mu_N\uparrow 2$ as $N\to\infty$, hence
$\mu_N^m\uparrow 2^m\ge\lambda>\lambda/2$.  Since $\mu_N<2$, we also
have $\mu_N^m<2^m\le\lambda$, so
$\alpha:=\lambda/\mu_N^m\in(1,2]$.
By Theorem~\ref{thm:KL}, there exists a closed subspace
$Y_0\subset Z_0:=\ell_\infty$ with $\lambda(Y_0,Z_0)=\alpha$ and
non-attainment.

Define recursively, for $k=1,\dots,m$,
\[
Z_k:=(Z_{k-1})_\infty^N,\qquad
Y_k:=\Sigma_N(Y_{k-1})\subset Z_k.
\]
By Lemma~\ref{lem:multiplication}, for each $k$,
$\lambda(Y_k,Z_k)=\mu_N\cdot\lambda(Y_{k-1},Z_{k-1})$,
and non-attainment is preserved.  Inductively,
$\lambda(Y_m,Z_m)=\mu_N^m\,\alpha=\lambda$,
and the infimum is not attained.

It remains to note that $Z_m$ is $1$-injective: $Z_0=\ell_\infty$ is
$1$-injective, and Lemma~\ref{lem:sum-injective} shows recursively that
each $Z_k$ is $1$-injective.
Lemma~\ref{lem:from-proj-to-inj}, applied to $Y_m\subset Z_m$, yields
that $Y_m$ is $(\lambda^+)$-injective but not $\lambda$-injective.
\end{proof}

\begin{remark}\label{rem:ambient_linf}
The ambient space $Z_m=(\ell_\infty)_\infty^{N^m}$ is isometrically
isomorphic to $\ell_\infty$.  Indeed, partition $\N$ into $N^m$
infinite sets $A_1,\dots,A_{N^m}$ and fix bijections
$\phi_j\colon \N\to A_j$.  Then
\[
J(x^{(1)},\dots,x^{(N^m)})(\phi_j(k)):=x^{(j)}_k
\qquad (j=1,\dots,N^m,\ k\in\N)
\]
defines a surjective isometry
$J\colon (\ell_\infty)_\infty^{N^m}\to \ell_\infty$.
Hence $Y_m$ may be viewed as a closed subspace of~$\ell_\infty$.
\end{remark}

\begin{remark}\label{rem:thm_A_summary}
Combining Theorem~\ref{thm:main_real} with
\cite[Theorem~A]{KaniaLewicki}, which treats $\lambda\in(1,2]$ via
the hyperplane construction, and the classical result of
Isbell--Semadeni \cite{IsbellSemadeni} at $\lambda=2$, we obtain
Theorem~A.  The proof is completely explicit once
Theorem~\ref{thm:KL} is taken as input: the only operation used to
pass from $(1,2]$ to arbitrary $\lambda>2$ is the zero-sum
construction $Y\mapsto\Sigma_N(Y)$.
\end{remark}

\section{An optimised quantitative decomposition: proof of Theorem~B}%
\label{sec:optimised}

\begin{theorem}\label{thm:optimised}
Let $X$ and $Y$ be Banach spaces such that
\begin{enumerate}[label=\textup{(\roman*)}]
\item there exist $1$-complemented subspaces $X'\subseteq X$ and
$Y'\subseteq Y$ with $X$ isometric to $Y'$ and $Y$ isometric to $X'$;
\item $X$ is isometric to $X\oplus_\infty X$ and $Y$ is isometric
to $Y\oplus_\infty Y$.
\end{enumerate}
Then $\dist(X,Y)\le 9+6\sqrt{3}$.
\end{theorem}

\begin{proof}
Fix norm-one projections $P\colon X\to X'$ and $R\colon Y\to Y'$, and
set $E:=\ker P$ and $F:=\ker R$.  By~(i), fix surjective isometries
$\theta\colon X\to Y'$ and $\eta\colon Y\to X'$.  By~(ii), fix
surjective isometries
$\varphi=(\varphi_1,\varphi_2)\colon X\to X\oplus_\infty X$ and
$\psi=(\psi_1,\psi_2)\colon Y\to Y\oplus_\infty Y$.

Let $a>0$ and put
\[
\mu:=\frac{1}{a},\qquad
\nu:=\frac{\sqrt{2a+1}}{a},\qquad
b:=\frac{a}{\sqrt{2a+1}}=\frac{1}{\nu}.
\]
Define three operators:
\begin{align*}
T_a\colon X&\to Y\oplus_\infty Y\oplus_\infty E,\\
T_a x&:=\bigl(\nu\,\psi_1\eta^{-1}Px,\;
\mu\,\psi_2\eta^{-1}Px,\; x-Px\bigr);\\[2mm]
S\colon Y\oplus_\infty Y\oplus_\infty E
&\to X\oplus_\infty X\oplus_\infty F,\\
S(y_1,y_2,e)&:=\bigl(\theta^{-1}Ry_1,\;
\eta y_2+e,\; y_1-Ry_1\bigr);\\[2mm]
U_a\colon X\oplus_\infty X\oplus_\infty F&\to Y,\\
U_a(x_1,x_2,f)&:=\theta\,\varphi^{-1}(ax_1,\,bx_2)+f.
\end{align*}
Each is bijective, with inverses:
\begin{align*}
T_a^{-1}(y_1,y_2,e)
&=\eta\,\psi^{-1}\!\bigl(\tfrac{1}{\nu}y_1,\,ay_2\bigr)+e,\\
S^{-1}(x_1,x_2,f)
&=\bigl(\theta x_1+f,\;\eta^{-1}Px_2,\;x_2-Px_2\bigr),\\
U_a^{-1}y
&=\bigl(\tfrac{1}{a}\,\varphi_1\theta^{-1}Ry,\;
\nu\,\varphi_2\theta^{-1}Ry,\;y-Ry\bigr).
\end{align*}
Set $W_a:=U_aST_a\colon X\to Y$.

\smallskip
\noindent\emph{Estimating $\|W_a\|$.}
Let $\|x\|\le 1$.  Since $\|Px\|\le 1$ and $\|x-Px\|\le 2$,
and $\|\psi_1\eta^{-1}Px\|\le 1$, while
$\|(I_Y-R)\psi_1\eta^{-1}Px\|\le 2$, a direct expansion yields
$\|W_ax\|\le\max\{a\nu,\,b(\mu+2)\}+2\nu$.
By the choice of parameters, $a\nu=\sqrt{2a+1}$ and
$b(\mu+2)=\sqrt{2a+1}$, so
$\|W_a\|\le 2\nu+\sqrt{2a+1}$.

\smallskip
\noindent\emph{Estimating $\|W_a\|$.}
Let $\|x\|\le 1$.  Writing out the composition, we obtain
\[
W_ax
=
\theta\varphi^{-1}\!\Bigl(
a\nu\,\theta^{-1}R\psi_1\eta^{-1}Px,\;
b\bigl(\mu\,\eta\psi_2\eta^{-1}Px+x-Px\bigr)
\Bigr)
+
\nu(I_Y-R)\psi_1\eta^{-1}Px.
\]
Hence
\begin{align*}
\|W_ax\|
&\le
\max\Bigl\{
a\nu\|R\psi_1\eta^{-1}Px\|,
\,b\bigl(\mu\|\psi_2\eta^{-1}Px\|+\|x-Px\|\bigr)
\Bigr\}
+\nu\|(I_Y-R)\psi_1\eta^{-1}Px\| \\
&\le \max\{a\nu,\,b(\mu+2)\}+2\nu,
\end{align*}
because $\|Px\|\le 1$, $\|x-Px\|\le 2$, $\|R\|=1$, and
$\|(I_Y-R)\|\le 1+\|R\|=2$.
By the choice of parameters,
\[
a\nu=\sqrt{2a+1}
\qquad\text{and}\qquad
b(\mu+2)=\sqrt{2a+1},
\]
so
\[
\|W_a\|\le 2\nu+\sqrt{2a+1}.
\]

\smallskip
\noindent\emph{Estimating $\|W_a^{-1}\|$.}
Let $\|y\|\le 1$.  Since $W_a^{-1}=T_a^{-1}S^{-1}U_a^{-1}$, we have
\[
U_a^{-1}y
=
\bigl(
\mu\,\varphi_1\theta^{-1}Ry,\;
\nu\,\varphi_2\theta^{-1}Ry,\;
y-Ry
\bigr),
\]
and therefore
\[
S^{-1}U_a^{-1}y
=
\bigl(
\mu\,\theta\varphi_1\theta^{-1}Ry+y-Ry,\;
\nu\,\eta^{-1}P\varphi_2\theta^{-1}Ry,\;
\nu(I_X-P)\varphi_2\theta^{-1}Ry
\bigr).
\]
Applying $T_a^{-1}$ gives
\[
W_a^{-1}y
=
\eta\psi^{-1}\!\Bigl(
\frac{1}{\nu}\bigl(\mu\,\theta\varphi_1\theta^{-1}Ry+y-Ry\bigr),\;
a\nu\,\eta^{-1}P\varphi_2\theta^{-1}Ry
\Bigr)
+
\nu(I_X-P)\varphi_2\theta^{-1}Ry.
\]
Consequently,
\begin{align*}
\|W_a^{-1}y\|
&\le
\max\Bigl\{
\frac{1}{\nu}(\mu+2),\;
a\nu
\Bigr\}
+2\nu \\
&=
\max\{b(\mu+2),\,a\nu\}+2\nu
=
2\nu+\sqrt{2a+1}.
\end{align*}
Thus $\|W_a^{-1}\|\le 2\nu+\sqrt{2a+1}$.

\smallskip
Setting $K(a):=2\nu+\sqrt{2a+1}=\frac{(a+2)}{a}\sqrt{2a+1}$,
we obtain $\dist(X,Y)\le K(a)^2=g(a)$ where
\[
g(a):=2a+9+\frac{12}{a}+\frac{4}{a^2}.
\]
Solving $g'(a)=0$ yields $a^3-6a-4=0$, whose unique positive root is
$a=1+\sqrt{3}$.  Substituting gives $g(1+\sqrt{3})=9+6\sqrt{3}$.
\end{proof}

\begin{corollary}\label{cor:linfty-ellinfty}
$\dist(L_\infty[0,1],\ell_\infty)\le 9+6\sqrt{3}$.
\end{corollary}

\begin{proof}
We verify the hypotheses of Theorem~\ref{thm:optimised}.
As noted in the proof of \cite[Corollary~5.2]{KorpalskiPlebanek},
both $L_\infty[0,1]$ and $\ell_\infty$ are $1$-injective and
isometrically square.

Moreover, that proof gives isometric embeddings in both directions:
$\ell_\infty$ embeds isometrically into $L_\infty[0,1]$ via disjoint
intervals, while $L_\infty[0,1]$ embeds isometrically into
$\ell_\infty$ because it has weak$^*$-separable dual ball.
Since the ambient spaces are $1$-injective, the images of these
embeddings are $1$-complemented. Thus the hypotheses of Theorem~\ref{thm:optimised} are satisfied, and
therefore
\(
\dist(L_\infty[0,1],\ell_\infty)\le 9+6\sqrt{3}.
\)
\end{proof}

\begin{remark}
The estimate in Theorem~\ref{thm:optimised} is obtained by optimising
an explicit one-parameter family of isomorphisms.  The only input
beyond the complemented embeddings is the square structure of the two
spaces.
\end{remark}

\subsection*{Acknowledgements}
The first-named author gratefully acknowledges support received from
NCN Sonata-Bis~13 (2023/50/E/ST1/00067).


\end{document}